\documentclass[11pt]{amsart}

\usepackage{amsmath,amssymb,amsthm,color,enumerate,comment,bm}
\usepackage{tikz}
\usepackage{todonotes}
\usepackage{url} 
\usepackage{hyperref} 
\usepackage{subcaption}

\newtheorem{thm}{Theorem}[section]
\newtheorem{prop}[thm]{Proposition}

\newtheorem{cor}[thm]{Corollary}

\newtheorem{theorem}[thm]{Theorem}

\theoremstyle{definition}

\newtheorem{prob}[thm]{Problem}

\theoremstyle{remark}

\newcommand{\FF}{\mathbb{F}}

\title[Non-uniqueness of optimal $1$-embeddings]{Non-uniqueness of optimal $1$-embeddings on surfaces up to weak equivalence}

\author[K. Enami]{Kengo Enami}
\address{College of Liberal Arts and Sciences,  Kitasato University, 1-15-1 Kitasato, Minami-ku, Sagamihara, Kanagawa 252-0373, Japan}
\email{enamikengo@gmail.com}



\begin{document}
\baselineskip.525cm

\begin{abstract}
We study the uniqueness of optimal $1$-embeddings on closed surfaces up to weak equivalence.
Answering a question of Suzuki in the negative, we show that for every closed surface other than the sphere, the projective plane and the Klein bottle, there exists an optimal $1$-embedded graph admitting two weakly inequivalent optimal $1$-embeddings.
Our construction is given explicitly on the torus and is then extended to other surfaces through a suitable connected-sum construction that preserves weak inequivalence.
As a consequence, we also obtain optimal $1$-embedded graphs with exponentially many pairwise weakly inequivalent optimal $1$-embeddings as the genus increases.
\end{abstract}

\maketitle

\section{Introduction}
In this paper, we only consider undirected, simple and connected graphs.
Ringel \cite{ringel} introduced the concept of ``$1$-plane graphs'' in the context of simultaneous vertex-face coloring of plane graphs.
In this paper, we focus on ``$1$-embeddings'' of graphs,
which generalize $1$-plane graphs to drawings on general surfaces.

A \emph{surface} means a compact connected
$2$-manifold without boundary.
Throughout this paper, all \emph{drawings} are assumed to be proper:
edges are simple arcs, adjacent edges do not cross, no edge passes through
a nonincident vertex, and no three edges meet at an interior point.
A drawing of a graph $G$ on a surface $\FF$ is called a \emph{$1$-embedding}
if each edge of $G$ is crossed at most once.
For a given graph $G$ and a given surface $\FF$,
if there exists a $1$-embedding of $G$ on $\FF$,
then we say that $G$ is \emph{$1$-embeddable} on $\FF$.
A graph $G$ is \emph{$1$-planar}
if $G$ is $1$-embeddable on the sphere (or equivalently, on the Euclidean plane).
An edge in a $1$-embedded graph on a surface is called a \emph{crossing edge} if it is involved in a crossing, and a \emph{non-crossing edge} otherwise.

It is well known that every $1$-planar graph of order $n$ has at most $4n-8$ edges and this bound is sharp (see, e.g., \cite{pt}).
It was also proved in \cite{nns,schumacher2} that on any surface $\FF$, the number of edges in a $1$-embeddable graph of order $n$  is at most $4n-4\chi(\FF)$, where $\chi(\FF)$ is the Euler characteristic of $\FF$.
A $1$-embedding on a surface is \emph{optimal} if it attains this bound.
It was shown in \cite{nns} that every simple optimal $1$-embedded graph can be obtained from a \emph{polyhedral} quadrangulation $H$,
that is, a $3$-connected and $3$-representative quadrangulation,
by adding a pair of crossing edges to each face of $H$.
We call $H$ the \emph{quadrangular subgraph}.

It is often convenient to regard a graph $G$ not only as a combinatorial object but also as a $1$-dimensional topological space, namely, a simplicial $1$-complex.
From this viewpoint, a $1$-embedding of $G$ on a surface $\FF$ can be regarded as a continuous map $f: G\to \FF$.
If $f$ is injective, then it is an embedding in the usual sense.
We call the image $f(G)\subset\FF$, together with its graph structure, a \emph{$1$-embedded graph}.
When no confusion can arise, we also use the term \emph{$1$-embedding}
for the image $f(G)$ itself, rather than for the map $f$.

When we deal with two or more $1$-embeddings of $G$ on a surface,
we denote them by $f_1(G), f_2(G), \dots$ to distinguish them.
Two $1$-embeddings $f_1,f_2:G\to\FF$ of a graph $G$ on a surface $\FF$
are said to be \emph{equivalent}
if there exists a homeomorphism $h:\FF\to\FF$ such that
$hf_1=f_2$, and they are \emph{inequivalent} otherwise.
We say that $G$ is \emph{uniquely $1$-embeddable} on $\FF$ (up to equivalence)
if any two $1$-embeddings of $G$ on $\FF$ are equivalent.

A weaker notion is \emph{weak equivalence}.
Two $1$-embeddings $f_1,f_2:G\to\FF$ are \emph{weakly equivalent}
if there exist an automorphism $\sigma:G\to G$ and a homeomorphism
$h:\FF\to\FF$ such that $hf_1=f_2\sigma$.
Otherwise, they are \emph{weakly inequivalent}.
Informally, weak equivalence allows us to ignore the labels of the vertices:
two $1$-embeddings with the same topological form are regarded as equivalent
even if their vertex labels differ by an automorphism of the graph.

In topological graph theory, much attention has been devoted to the uniqueness and re-embeddability of graph embeddings.
The \emph{re-embeddability} problem asks how many inequivalent embeddings a given graph admits on a surface and how these embeddings are related to one another.
This aspect of graph embeddings is also referred to as their \emph{flexibility}.

Classical results in this direction are Whitney's theorems
\cite{whitney1,whitney2}.
They state that every $3$-connected planar graph is uniquely embeddable on the sphere,
and that any two embeddings of a $2$-connected planar graph can be transformed into each other by a sequence of local operations known as \emph{Whitney flips}.
Related problems have also been studied extensively on other surfaces (see, e.g., \cite{enami,mrss,mohar,mr1,rzz,st,suzukiReem}).

Naturally, questions of uniqueness and re-embeddability also arise for $1$-embeddings.
The re-embeddability of $1$-embeddings is sometimes referred to as
\emph{re-$1$-embeddability} in the literature.
Schumacher~\cite{schumacher1} proved that every $5$-connected optimal
$1$-plane graph is uniquely $1$-embeddable on the sphere, except for a certain
infinite family.
Suzuki~\cite{suzukimax} later showed that the $5$-connectivity assumption is unnecessary.
Suzuki~\cite{suzukipro} subsequently studied the corresponding problem for
$1$-embeddings on the projective plane.

\begin{theorem}[Suzuki \cite{suzukipro}]\label{thm:1embeddigPro}
The number of inequivalent $1$-embeddings of an optimal $1$-embedded graph
on the projective plane whose quadrangular subgraph is bipartite is at most $24$.
\end{theorem}

The situation is quite different when the quadrangular subgraph is non-bipartite.
In this case, Suzuki~\cite{suzukipro} constructed optimal $1$-embedded graphs
on the projective plane with arbitrarily many inequivalent $1$-embeddings.

By contrast, when weak equivalence is considered, uniqueness appears to persist.
Every optimal $1$-plane graph is known to be uniquely
$1$-embeddable up to weak equivalence.
Moreover, in Suzuki's examples on the projective plane, the multiple inequivalent $1$-embeddings are all weakly equivalent.
Motivated by these observations, Suzuki posed the following question.

\begin{prob}[Suzuki \cite{suzukipro}]\label{prob:weak}
Is every optimal $1$-embedded graph on a surface $\FF$ uniquely $1$-embeddable
into $\FF$ up to weak equivalence?
If not, give an example admitting two weakly inequivalent $1$-embeddings.
\end{prob}

We answer Problem~\ref{prob:weak} in the negative.
In fact, this non-uniqueness is a widespread phenomenon:
it occurs on every orientable closed surface of positive genus and on every nonorientable closed surface of genus at least three.
More precisely, we prove the following.

\begin{thm}\label{thm:main}
For every closed surface $\FF$ other than the sphere, the projective plane and the Klein bottle,
there exists an optimal $1$-embedded graph on $\FF$ that admits two weakly inequivalent optimal $1$-embeddings into $\FF$.
\end{thm}

In Section~\ref{sec:torus}, we give an explicit construction on the torus.
In Section~\ref{sec:general}, we modify and extend this construction to obtain
the result for the other surfaces covered by Theorem~\ref{thm:main}.

\section{Toroidal case}\label{sec:torus}

We first establish the toroidal case of Theorem~\ref{thm:main}.

\begin{prop}\label{prop:torus}
There exists an optimal $1$-embedded graph on the torus that admits two weakly inequivalent optimal $1$-embeddings.
\end{prop}

\begin{proof}

    \begin{figure}[tbp]
    \centering
    \begin{minipage}{0.45\linewidth}
        \centering
        \includegraphics[width=0.95\linewidth]{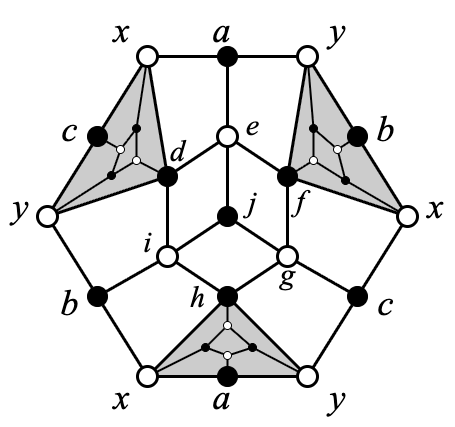}
    \end{minipage}
    \hspace{0.05\linewidth}
    \begin{minipage}{0.45\linewidth}
        \centering
        \includegraphics[width=0.95\linewidth]{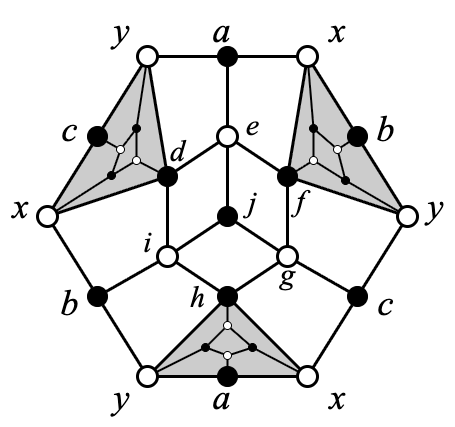}
    \end{minipage}
    \caption{Two weakly equivalent quadrangulations on the torus}
    \label{fig:beforeAndAfter}
    \end{figure}

Figure~\ref{fig:beforeAndAfter} shows two quadrangulations embedded on the torus.
The torus itself is represented by a hexagon in which opposite sides are identified in the same direction.
That is, the boundary vertices with the same label are identified with each other.
It is straightforward to verify that these quadrangulations are
$3$-connected and $3$-representative, that is, polyhedral.
These two embeddings have the same shape, but some of the vertex labels are different.

To see the correspondence between them, consider the three shaded quadrangle patches in the left-hand embedding.
These patches may be regarded as being hinged at the vertex pairs
$\{a,h\}$, $\{b,f\}$, and $\{c,d\}$, respectively.
Simultaneously flipping the three patches about these hinges yields the embedding shown on the right.
During this operation, the six hinge vertices $a,b,c,d,f$ and $h$ remain fixed,
while the positions of $x$ and $y$ are interchanged.
Note that each shaded patch contains several unlabeled interior vertices,
which are carried along with the corresponding patch during the flip.

This flip preserves the graph structure.
Indeed, apart from the unlabeled vertices lying inside the shaded patches,
every vertex adjacent to $x$ or $y$ is one of
$a,b,c,d,f$ and $h$,
and each of these six vertices is adjacent to both $x$ and $y$.
Together with the correspondence of the interior vertices induced by the flip,
this shows that the two quadrangulations are isomorphic as graphs.
Let $Q$ denote this graph, and let $f_1(Q)$ and $f_2(Q)$ denote
the left-hand and right-hand embeddings in Figure~\ref{fig:beforeAndAfter}, respectively.
Thus, $f_1(Q)$ and $f_2(Q)$ are weakly equivalent.

Under the re-embedding from $f_1(Q)$ to $f_2(Q)$,
the six facial cycles
$xaed$, $yfea$, $xcgf$, $yhgc$, $xbih$ and $ydib$
in $f_1(Q)$ are no longer facial in $f_2(Q)$;
in fact, they become essential cycles on the torus.
On the other hand, every other facial cycle of $f_1(Q)$ remains facial in $f_2(Q)$.

    \begin{figure}[tbp]
        \centering
        \includegraphics[width=0.7\linewidth]{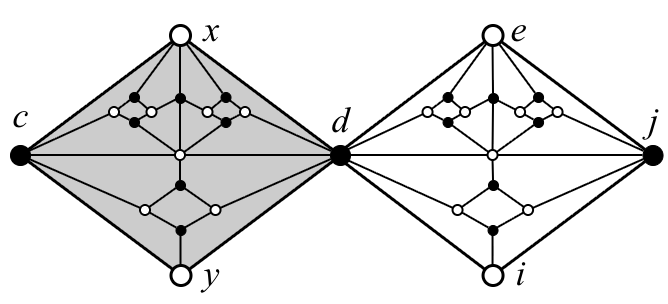}
        \caption{Replacement subgraphs}
        \label{fig:part}
    \end{figure}

We next modify $Q$ locally.
In $f_1(Q)$, replace the subgraphs contained in the quadrangle regions
bounded by $xcyd$ and $dije$ with the graphs shown in Figure~\ref{fig:part},
in exactly the orientations shown there.
In $f_2(Q)$, the replacement subgraph for the region bounded by $dije$ is inserted in the same way,
whereas the replacement subgraph for the region bounded by $xcyd$ is inserted after being flipped about the hinge at $c$ and $d$.
It is easy to check that the resulting quadrangulations are still
polyhedral.
The two resulting quadrangulations are again isomorphic as graphs;
let $\tilde{Q}$ denote this graph.
We denote the embeddings corresponding to $f_1(Q)$ and $f_2(Q)$ by
$\tilde{f}_1(\tilde{Q})$ and $\tilde{f}_2(\tilde{Q})$, respectively.
As in Figure~\ref{fig:beforeAndAfter},
$\tilde{f}_2(\tilde{Q})$ is obtained from $\tilde{f}_1(\tilde{Q})$
by the same flip operation described above.

In fact, $\tilde{f}_1(\tilde{Q})$ and $\tilde{f}_2(\tilde{Q})$ are weakly inequivalent.
More precisely, in $\tilde{Q}$, the vertices $d,e,x$ and $y$
have degrees $10,7,11$ and $9$, respectively,
while every other vertex has degree at most $6$.
Hence every automorphism of $\tilde{Q}$ fixes
$d,e,x$ and $y$ individually.
In $\tilde{f}_1(\tilde{Q})$, the cycle $xdea$ is facial and contains
the three vertices $x,d$ and $e$, which are fixed by every automorphism of $\tilde{Q}$.
On the other hand, no facial cycle of $\tilde{f}_2(\tilde{Q})$
contains all three of these vertices.
It follows that the two embeddings are weakly inequivalent.

We now construct a pair of weakly inequivalent optimal 1-embeddings from
$\tilde{f}_1(\tilde{Q})$ and $\tilde{f}_2(\tilde{Q})$.
For each of these quadrangulations, add a pair of crossing edges to every quadrangle face.
This yields optimal 1-embeddings whose quadrangular subgraphs are
$\tilde{f}_1(\tilde{Q})$ and $\tilde{f}_2(\tilde{Q})$, respectively.

Recall that, among the facial cycles of $\tilde{f}_1(\tilde{Q})$,
only the six cycles $xaed$, $yfea$, $xcgf$, $yhgc$, $xbih$, and $ydib$ are not facial in $\tilde{f}_2(\tilde{Q})$.
Moreover, each of the six vertices $a,b,c,d,f$ and $h$ is adjacent to both $x$ and $y$.
After adding the crossing edges, each of $e,g$ and $i$ is also joined to both $x$ and $y$ by crossing edges in both embeddings.
Consequently, the same set of edges is obtained from
$\tilde{f}_1(\tilde{Q})$ and $\tilde{f}_2(\tilde{Q})$ after adding the crossing edges.
Thus, these two resulting optimal 1-embeddings are drawings of the same graph.
Let this graph be denoted by $G$, and denote these two optimal 1-embeddings by $g_1(G)$ and $g_2(G)$, respectively.

Finally, $g_1(G)$ and $g_2(G)$ are weakly inequivalent.
Indeed, if these two optimal $1$-embeddings were weakly equivalent,
then the corresponding homeomorphism would preserve edge crossings.
Hence the automorphism involved in the weak equivalence would map crossing edges
to crossing edges and non-crossing edges to non-crossing edges,
and would therefore induce a weak equivalence between their quadrangular subgraphs
$\tilde{f}_1(\tilde{Q})$ and $\tilde{f}_2(\tilde{Q})$.
This contradicts the weak inequivalence of
$\tilde{f}_1(\tilde{Q})$ and $\tilde{f}_2(\tilde{Q})$ established above.
Therefore, $g_1(G)$ and $g_2(G)$ are weakly inequivalent optimal
$1$-embeddings of the same graph $G$.

\end{proof}

\section{General case}\label{sec:general}

We now extend the construction in the toroidal case to an arbitrary closed surface other than the sphere, the projective plane, and the Klein bottle.

\begin{proof}[Proof of Theorem~\ref{thm:main}]
The toroidal case has already been established in Proposition~\ref{prop:torus}.
We therefore assume that $\FF$ is not the torus.

We first modify the toroidal construction given in Section~\ref{sec:torus}.
Recall the quadrangulation $\tilde{Q}$ and its two embeddings
$\tilde{f}_1(\tilde{Q})$ and $\tilde{f}_2(\tilde{Q})$
constructed in the previous section.
The subgraphs inside the quadrangle regions bounded by $xcyd$ and $edij$
can be further enlarged without affecting the flip operation that transforms
$\tilde{f}_1(\tilde{Q})$ into $\tilde{f}_2(\tilde{Q})$.

For this purpose, we use the following operation.
Given a quadrangle face $v_1v_2v_3v_4$, a \emph{cube addition} consists of
inserting four new vertices $u_1,u_2,u_3,u_4$ inside the face,
adding the cycle $u_1u_2u_3u_4$, and adding the four edges
$v_iu_i$ for $i=1,2,3,4$.
This operation preserves polyhedrality.
By repeatedly applying cube additions to suitable faces inside the two regions,
we can make the minimum of the degrees of $x,d$, and $e$
arbitrarily larger than the maximum degree of all other vertices.

Let $Q^*$ denote the resulting graph, and let
$f^*_1(Q^*)$ and $f^*_2(Q^*)$ denote the embeddings obtained from
$\tilde{f}_1(\tilde{Q})$ and $\tilde{f}_2(\tilde{Q})$, respectively.
The same flip operation as in Section~\ref{sec:torus} transforms
$f^*_1(Q^*)$ into $f^*_2(Q^*)$.

These two embeddings are still weakly inequivalent.
Indeed, $x,d$, and $e$ are the only vertices of sufficiently large degree,
and hence every automorphism of $Q^*$ preserves the set $\{x,d,e\}$.
In $f^*_1(Q^*)$, the cycle $xdea$ is facial and contains all three vertices
$x,d$, and $e$.
On the other hand, no facial cycle of $f^*_2(Q^*)$ contains all three of them.
Therefore,
$f^*_1(Q^*)$ and $f^*_2(Q^*)$ are weakly inequivalent.

We now extend this construction to $\FF$.
By the classification of closed surfaces, there exists a closed surface $\FF_0$
such that
\[
    \FF \cong T^2 \# \FF_0.
\]
Choose a polyhedral quadrangulation $R$ of $\FF_0$ and a quadrangle face $C$ of $R$,
and let $\Delta$ be the maximum degree of $R$.
We may choose the cube additions so that
the minimum degree among $x,d$, and $e$
exceeds the maximum degree of all other vertices of $Q^*$ by more than $\Delta$.

Remove the interiors of the face bounded by $efgh$ in $f^*_1(Q^*)$
and of the face $C$ in $R$, and identify their boundary $4$-cycles.
This produces a polyhedral quadrangulation of
$T^2\#\FF_0\cong\FF$.
Indeed, this face-sum preserves both $3$-connectivity and $3$-representativity.
Performing the same operation with $f^*_2(Q^*)$ produces another embedding
of the same graph.
Let $Q_{\FF}$ denote this graph, and denote these two embeddings by
\[
    f^{\FF}_1(Q_{\FF})
    \quad\text{and}\quad
    f^{\FF}_2(Q_{\FF}),
\]
respectively.
The attachment of $R$ does not interfere with the flip operation.
Indeed, the face bounded by $efgh$ is unchanged by the re-embedding of $Q^*$,
and the same quadrangulation $R$ is attached along this face in both embeddings.
Hence the flip can still be performed entirely in the toroidal part,
with $R$ left unchanged.
Thus,
$f^{\FF}_2(Q_{\FF})$ is obtained from
$f^{\FF}_1(Q_{\FF})$ by the same flip operation as before.

We next verify that these two embeddings are weakly inequivalent.
If a boundary vertex $v$ of $efgh$ is identified with a boundary vertex $u$ of $C$,
then its degree increases by $\deg_R(u)-2$, which is at most $\Delta-2$.
The degrees of $x$ and $d$ are unchanged, since neither vertex lies on $efgh$,
while the degree of $e$ can only increase.
By the choice of the degree gap in $Q^*$,
$x,d$, and $e$ are therefore still the only vertices of sufficiently large degree
in $Q_{\FF}$.
Hence every automorphism of $Q_{\FF}$ preserves the set $\{x,d,e\}$.

The attachment along $efgh$ does not affect the facial cycle $xdea$.
Thus, $f^{\FF}_1(Q_{\FF})$ has a facial cycle containing all three vertices
$x,d$, and $e$, whereas $f^{\FF}_2(Q_{\FF})$ has no such facial cycle.
Therefore,
$f^{\FF}_1(Q_{\FF})$ and $f^{\FF}_2(Q_{\FF})$ are weakly inequivalent.

Finally, add a pair of crossing edges to every face of
$f^{\FF}_1(Q_{\FF})$ and $f^{\FF}_2(Q_{\FF})$.
This gives two optimal $1$-embeddings on $\FF$.
As in the toroidal construction, the added crossing edges in the $Q^*$-part
correspond under the re-embedding, while those in the $R$-part are identical
in the two embeddings.
Hence the two resulting optimal $1$-embeddings are embeddings of the same graph.
Let this graph be denoted by $G_{\FF}$, and denote the two embeddings by
\[
    g^{\FF}_1(G_{\FF})
    \quad\text{and}\quad
    g^{\FF}_2(G_{\FF}).
\]

These two optimal $1$-embeddings are weakly inequivalent.
Indeed, any weak equivalence between them would preserve edge crossings
and hence induce a weak equivalence between their quadrangular subgraphs
$f^{\FF}_1(Q_{\FF})$ and $f^{\FF}_2(Q_{\FF})$,
contrary to what was proved above.
Therefore, $G_{\FF}$ admits two weakly inequivalent optimal
$1$-embeddings into $\FF$.

\end{proof}

The proof of Theorem~\ref{thm:main} also gives a simple way to construct graphs with many pairwise weakly inequivalent optimal $1$-embeddings.
We record this consequence below.

\begin{cor}\label{cor:many}
The following statements hold.
\begin{enumerate}
    \item For every $g\geq 1$, there exists an optimal $1$-embedded graph
    on the orientable surface of genus $g$ that admits at least $2^g$
    pairwise weakly inequivalent optimal $1$-embeddings.

    \item For every $k\geq 3$, there exists an optimal $1$-embedded graph
    on the nonorientable surface of genus $k$ that admits at least
    \[
        2^{\lfloor (k-1)/2\rfloor}
    \]
    pairwise weakly inequivalent optimal $1$-embeddings.
\end{enumerate}
\end{cor}

\begin{proof}
The construction used in the proof of Theorem~\ref{thm:main} can be
repeated independently in several toroidal parts.
For the $j$th toroidal part, let
\[
    X_j=\{x_j,d_j,e_j\}
\]
denote its three distinguished vertices.
We choose the degree parameters in the different toroidal parts sufficiently
far apart so that every automorphism of the resulting graph preserves each
set $X_j$.
As in the proof of Theorem~\ref{thm:main}, one of the two embeddings of the
$j$th toroidal part has a facial cycle containing all three vertices of
$X_j$, whereas the other has no such facial cycle.
The face-sums do not affect this distinction.
Hence, if two choices of re-embeddings differ in at least one toroidal part,
the resulting embeddings cannot be weakly equivalent.

For the orientable surface of genus $g$, take $g$ copies of the toroidal
construction and join them successively along suitable quadrangle faces.
The resulting surface is the connected sum of $g$ tori.
Since the two re-embeddings can be chosen independently in each of the
$g$ toroidal parts, this gives $2^g$ pairwise weakly inequivalent
quadrangulations.

For the nonorientable case, let $N_k$ denote the nonorientable surface of
genus $k$.
If $k=2r+1$, start with a polyhedral quadrangulation of the projective plane
and attach $r$ copies of the toroidal construction.
Since
\[
    N_1\# rT^2\cong N_{2r+1},
\]
this gives $2^r$ pairwise weakly inequivalent quadrangulations on $N_k$.
If $k=2r+2$, start instead with a polyhedral quadrangulation of the Klein bottle
and attach $r$ copies of the toroidal construction.
Since
\[
    N_2\# rT^2\cong N_{2r+2},
\]
we again obtain $2^r$ pairwise weakly inequivalent quadrangulations.
In both cases,
\[
    r=\left\lfloor\frac{k-1}{2}\right\rfloor.
\]

Finally, adding a pair of crossing edges to every quadrangle face,
as in the proof of Theorem~\ref{thm:main}, converts these quadrangulations
into optimal $1$-embeddings of the same graph.
Any weak equivalence between two of these optimal $1$-embeddings would induce
a weak equivalence between their quadrangular subgraphs.
Hence the resulting optimal $1$-embeddings are also pairwise weakly inequivalent.
\end{proof}

\section{Remarks}\label{sec:remarks}
We conclude with some remarks on the cases not covered by
Theorem~\ref{thm:main} and on the number of weakly inequivalent optimal
$1$-embeddings.

Theorem~\ref{thm:main} leaves the projective plane and the Klein bottle
unresolved.
Suzuki~\cite{suzukimax} showed that every optimal
$1$-plane graph is uniquely $1$-embeddable up to weak equivalence.
It remains open whether there exists an optimal $1$-embedded graph on the
projective plane or the Klein bottle that admits two weakly inequivalent
optimal $1$-embeddings.

It is also natural to ask how many weakly inequivalent optimal
$1$-embeddings a graph can have on a fixed surface.
Corollary~\ref{cor:many} shows that the number can grow exponentially with
the genus when the surface is allowed to vary.
A related question concerning equivalence, rather than weak equivalence,
was posed by Suzuki~\cite{suzukipro}.
This leads to the following question.

\begin{prob}\label{prob:number}
For a fixed closed surface $\FF$, is there a constant $c(\FF)$ such that
every optimal $1$-embedded graph on $\FF$ admits at most $c(\FF)$
pairwise weakly inequivalent optimal $1$-embeddings?
\end{prob}

For the sphere, the answer to Problem~\ref{prob:number} is affirmative
with $c(\FF)=1$ by Suzuki~\cite{suzukimax}.
For the projective plane, the examples in~\cite{suzukipro} with arbitrarily
many inequivalent optimal $1$-embeddings do not give an unbounded number of
weak equivalence classes, since all of those embeddings are weakly equivalent.
Thus, Problem~\ref{prob:number} remains open even for the projective plane.

\section*{Acknowledgements}
This work was supported by JSPS KAKENHI Grant Number JP23K13006.

\section*{Data availability}
No data were used for the research described in the article.


\end{document}